\documentclass[a4paper,11pt]{article}
\usepackage{color}

\usepackage[english]{babel}
\usepackage{amsmath,amsthm,amssymb,mathrsfs,latexsym,amsfonts}
\usepackage{graphicx,psfrag,epsfig}
\usepackage{bbm}
\usepackage{a4wide}
\usepackage{authblk}

\numberwithin{equation}{section}
\usepackage{latexsym}
\newtheorem{theorem}{Theorem}[section]
\newtheorem{lemma}[theorem]{Lemma}
\newtheorem{proposition}[theorem]{Proposition}

\newtheorem{definition}{Definition\rm}

\newcounter{paraga}[section]

\usepackage[latin1]{inputenc}

\newcommand{\N}{\mathbb{N}}
\newcommand{\Z}{\mathbb{Z}}
\newcommand{\Q}{\mathbb{Q}}
\newcommand{\R}{\mathbb{R}}
\newcommand{\C}{\mathbb{C}}

\renewcommand{\a}{\alpha}
\renewcommand{\o}{\omega}

\begin{document}

\def\MP{\,{<\hspace{-.5em}\cdot}\,}
\def\SP{\,{>\hspace{-.3em}\cdot}\,}
\def\PM{\,{\cdot\hspace{-.3em}<}\,}
\def\PS{\,{\cdot\hspace{-.3em}>}\,}
\def\EP{\,{=\hspace{-.2em}\cdot}\,}
\def\PP{\,{+\hspace{-.1em}\cdot}\,}
\def\PE{\,{\cdot\hspace{-.2em}=}\,}
\def\N{\mathbb N}
\def\C{\mathbb C}
\def\Q{\mathbb Q}
\def\R{\mathbb R}
\def\T{\mathbb T}
\def\A{\mathbb A}
\def\Z{\mathbb Z}
\def\demi{\frac{1}{2}}

\newtheorem{Main}{Theorem}
\newtheorem{Coro}{Corollary}
\newtheorem*{Principal}{Theorem}

\renewcommand{\theMain}{\Alph{Main}}
\renewcommand{\theCoro}{\Alph{Coro}}

\setcounter{tocdepth}{3}

\begin{titlepage}
\author{Abed Bounemoura\footnote{CNRS - CEREMADE - IMCCE/ASD, abedbou@gmail.com} {} and Bassam Fayad\footnote{CNRS - IMJ-PRG and Centro Ennio De Giorgi, bassam.fayad@imj-prg.fr} {} and Laurent Niederman\footnote{Laboratoire Math\'ematiques d'Orsay \& IMCCE/ASD, Laurent.Niederman@math.u-psud.fr}}
\title{\LARGE{\textbf{Double exponential stability of quasi-periodic motion in Hamiltonian systems}}}
\end{titlepage}

\maketitle

\begin{abstract}
We prove that generically, both in a topological and measure-theoretical sense, an invariant Lagrangian Diophantine torus of a Hamiltonian system is doubly exponentially stable in the sense that nearby solutions remain close to the torus for an interval of time which is doubly exponentially large with respect to the inverse of the distance to the torus. We also prove that for an arbitrary small perturbation of a generic integrable Hamiltonian system, there is a set of almost full positive Lebesgue measure of KAM tori which are doubly exponentially stable. Our results hold true for real-analytic but more generally for Gevrey smooth systems.
\end{abstract}
 
\section{Introduction and results}
\newcommand{\marge}[1]{\marginpar{\tiny #1}}

\subsection{Introduction}

The goal of this paper is to prove that invariant Lagrangian Diophantine tori in Hamiltonian systems are generically doubly exponentially stable. We will consider two settings.

In a first setting, we consider a Lagrangian Diophantine torus invariant by a Hamiltonian flow on an arbitrary symplectic manifold. It is well-known that by a symplectic change of coordinates, one can consider a Hamiltonian
\begin{equation}\label{H1}\tag{H}
H(\theta,I)=\omega\cdot I + O\left(||I||^2\right), \quad (\theta,I)\in \T^n \times B,
\end{equation} 
where $\,\cdot\,$ denotes the Euclidean scalar product, $||\,\cdot\,||$ is the associated norm, $n \geq 1$ is an integer, $\T^n:=\R^n / \Z^n$, $B$ is some open bounded neighborhood of the origin in $\R^n$ and $\omega\in \R^n$ is a vector which is assumed to be Diophantine: there exist constants $0<\gamma \leq 1$ and $\tau \geq n-1$ such that for any $k=(k_1,\dots,k_n) \in \Z^n\setminus \{0\}$,
\begin{equation}\label{dio}\tag{Dio$_{\gamma,\tau}$}
|k\cdot\omega|\geq \gamma|k|^{-\tau}, \quad |k|:=|k_1|+\cdots+|k_n|. 
\end{equation} 
The torus $\mathcal{T}_\omega:=\T^n \times \{0\}$ is then invariant by the Hamiltonian flow of $H$, and the flow restricted on $\mathcal{T}_\omega$ is quasi-periodic with frequency $\omega$. There are several questions one can ask about the stability or instability properties of such an invariant torus, for instance:
\begin{itemize}
\item[(1)] Is $\mathcal{T}_\omega$ accumulated by a large set of Lagrangian invariant tori?
\item[(2)] Is $\mathcal{T}_\omega$ topologically unstable? Recall that the invariant torus $\mathcal{T}_\omega$ is said to be topologically stable if it admits a basis of neighborhoods which are positively invariant by the Hamiltonian flow, and that it is topologically unstable if it is not topologically stable. 
\item[(3)] Given an arbitrary $r>0$ small enough and an arbitrary solution starting in the $r$-neighborhood of $\mathcal{T}_\omega$, how large is the ``stability" time $T(r)$ during which the solution remains in the $2r$-neighborhood of $\mathcal{T}_\omega$?  
\end{itemize}
These questions are related to, respectively, KAM theory, Arnold diffusion and Nekhoroshev theory. Concerning $(1)$, it is well-known that if $H$ is sufficiently smooth, under a generic condition, $\mathcal{T}_\omega$ is accumulated by a set of Lagrangian invariant tori of positive Lebesgue measure, which has Lebesgue density one at $\mathcal{T}_\omega$. Assuming $H$ to be real-analytic, without further assumption it is conjectured (see \cite{Her98}) that $\mathcal{T}_\omega$ is accumulated by a set of Lagrangian invariant tori of positive Lebesgue measure (see \cite{EFK15} for some partial results). Concerning $(2)$, almost nothing is known in the real-analytic category : for instance, it is still not known if there exists a real-analytic torus $\mathcal{T}_\omega$ which is topologically unstable. In the smooth category, Douady (\cite{Dou88}) gave examples of topologically unstable invariant tori with any given Birkhoff normal form at the invariant torus. Degenerate examples having almost all orbits oscillating between the neighborhood of the invariant Diophantine torus and infinity were constructed in \cite{EFK15}. As for the generic behavior, even in the smooth category this is still an open problem (see \cite{GK14} for some partial result in the case $n=3$).

Finally, concerning $(3)$, it is well-known that for real-analytic systems, $T(r)$ is always at least exponentially large, more precisely it is of order $\exp\left(r^{-a}\right)$ with the exponent $a=(1+\tau)^{-1}$. A similar result holds true for Gevrey smooth systems (with the exponent $\alpha^{-1}a$ instead of $a$, where $\alpha\geq 1$ is the Gevrey exponent). Moreover, in the real-analytic case and under some convexity assumption (which is non-generic), it was proved in \cite{MG95a} that $T(r)$ is at least doubly exponentially large, that is, it is of order $\exp\exp\left(r^{-a}\right)$. Our first theorem extends this last result: for real-analytic Hamiltonians, or more generally Gevrey smooth Hamiltonians, under a generic condition, $T(r)$ is at least doubly exponentially large; see Theorem~\ref{Th1} below and its corollary for precise statements. This result is the counterpart to a result we previously obtained in the context of elliptic equilibrium points (see \cite{BFN15}).

In a second setting, we consider a small $\varepsilon$-perturbation of an integrable Hamiltonian system in action-angle coordinates, namely 
\begin{equation}\label{H2}\tag{H$_\varepsilon$}
H_\varepsilon(\theta,I)=h(I)+\varepsilon f(\theta,I), \quad (\theta,I)\in \T^n \times \bar{D},
\end{equation}   
where $\bar{D}$ is the closure of some open bounded domain $D \subseteq \R^n$, and $\varepsilon\geq 0$ is a small parameter. When $\varepsilon=0$, letting $\omega_*=\nabla h(I_*) \in \R^n$ for $I_* \in D$, the Lagrangian tori $\mathcal{T}_{\omega_*}:=\{I=I_*\}$ are obviously invariant by the Hamiltonian flow of $H_0=h$, and they are quasi-periodic with frequency $\omega_*$. Assuming that $H_\varepsilon$ is sufficiently smooth and $\varepsilon$ sufficiently small, under the Kolmogorov non degeneracy condition (a generic condition on $h$, see Definition~\ref{defK} below) the classical KAM theorem asserts that most of the invariant tori of the integrable system persist under perturbation. More precisely, let $\Omega$ be the image by $\nabla h$ of $\bar{D}$ and, fixing $0<\gamma \leq 1$, $\tau > n-1$, let  
\begin{equation} \label{def.Omega}
\Omega_{\gamma,\tau}:= \left\{  \omega \in \Omega \, | \, d(\omega,\partial \Omega) \geq \gamma, \; \omega \in {\rm Dio}_{\gamma,\tau} \right\}. 
\end{equation} 
Observe that since $\tau>n-1$, the complement of $\Omega_{\gamma,\tau}$ in $\Omega$ has a measure of order $\gamma$. The KAM theorem asserts that if one assumes that $\varepsilon \leq c \gamma^2$ where $c$ is some positive constant that depends on $h$ and $f$, then there exists a set
\[ \mathcal{K}=\bigcup_{\omega_* \in \Omega_{\gamma,\tau}}\mathcal{T}_{\omega_*}^\varepsilon \subseteq \T^n \times D \]
whose complement has a Lebesgue measure of order at most $\gamma$, and which consists of Lagrangian Diophantine tori $\mathcal{T}_{\omega_*}^\varepsilon$ invariant by $H_\varepsilon$ that converge to $\mathcal{T}_{\omega_*}$ as $\varepsilon$ goes to zero. For real-analytic systems, this is a classical result (see \cite{Pos01} and references therein) and for Gevrey smooth systems, this result is due to Popov (\cite{Pop04}). Our second theorem states that for real-analytic or Gevrey smooth systems, under a further generic condition on $h$ and a further smallness assumption on $\varepsilon$, a sub-family $\mathcal{K}'$ of this family $\mathcal{K}$ of KAM tori (with a complementary still having measure of order at most $\gamma$) is doubly exponentially stable. We refer to Theorem~\ref{Th2} below and its corollary for precise statements. This result extends a previous result obtained in \cite{MG95b}, \cite{GM97} in the real-analytic case, under the stronger (and non-generic) condition that $h$ is quasi-convex.  

\subsection{Results}

Let us now state more precisely our results. Given some bounded domain $U \subseteq \R^n$, and real numbers $\alpha \geq 1$, $\beta\geq 1$, $L_1>0$ and $L_2>0$, we define $G^{\alpha,\beta}_{L_1,L_2}(\T^n \times U)$ to be the space of smooth functions $H$ on $\T^n \times U$ such that
\begin{equation*}
||H||_{\alpha,\beta,L_1,L_2,U}:=\sup_{(k,l)\in \N^{2n}}\sup_{(\theta,I)\in \T^n \times U}\left(|\partial_\theta^k\partial_I^lH(\theta,I)|L_1^{-|k|}L_2^{-|l|}k!^{-\alpha}l!^{-\beta}\right)<\infty.
\end{equation*}
We shall always assume that $L_1 \geq 1$ and $L_2 \geq 1$. This space $G^{\alpha,\beta}_{L_1,L_2}(\T^n \times U)$, equipped with the above norm, is a Banach space. When $\alpha=\beta$ and $L_1=L_2=L$, we simply write $G^{\alpha}_{L}(\T^n \times U)$ and when $\alpha=1$, then $G^{1}_{L}(\T^n \times U)$ is the space of real-analytic functions which can be extended as a holomorphic function on a complex $s$-neighborhood of $\T^n \times U$ in $\C^n/\Z^n \times \C^n$, with $s<L^{-1}$. For simplicity, we shall also say that a function is $(\alpha,\beta)$-Gevrey (respectively $\alpha$-Gevrey) if it belongs to $G^{\alpha,\beta}_{L_1,L_2}(\T^n \times U)$ (respectively to $G^{\alpha}_{L}(\T^n \times U)$). 

Let us introduce the following notations: given some $\rho>0$, some domain $U \subseteq \R^n$ and some subset $S \subset \T^n \times U$, we denote by 
\[ V_\rho U:=\{I \in \R^n \; | \; ||I-U||<\rho  \}, \quad V_\rho S:=\{(\theta,I) \in \T^n \times \R^n \; | \; ||(\theta,I)-S||<\rho  \} \]
the open $\rho$-neighborhood of $U$ in $\R^n$ and of $S$ in $\T^n \times \R^n$.
  
\begin{definition} 
An invariant quasi-periodic Lagrangian torus $\mathcal{T}$ embedded in $\T^n \times U$ is {\it doubly exponentially stable with exponent $u>0$} if there exist positive constants $r_*$ and $C$ such that for any $r\leq r_*$ and any solution $(\theta(t),I(t))$ of the Hamiltonian system associated to $H$ with $(\theta(0),I(0)) \in V_r\mathcal{T}$, we have
\[ (\theta(0),I(0)) \in V_{2r}\mathcal{T}, \quad \forall |t| \leq \exp\left(\exp \left(C r^{-u}\right)\right). \] 
\end{definition}

Now let us come back to our first setting, that is we consider a Hamiltonian as in~\eqref{H1}, with $\omega$ satisfying~\eqref{dio}. In this setting, the torus $\mathcal{T}_\omega$ is doubly exponentially stable with exponent $u>0$ if there exist positive constants $r_*$ and $C$ such that for any $r\leq r_*$ and any solution $(\theta(t),I(t))$ of the Hamiltonian system associated to $H$ with $||I(0)||\leq r$, we have
\[ ||I(t)||\leq 2r, \quad \forall |t| \leq \exp\left(\exp \left(C r^{-u}\right)\right). \] 

If $H$ is smooth, given any integer $m \geq 2$, there exists a symplectic transformation $\Phi_m$ defined on a neighborhood of $\T^n \times \{0\}$ which is close to the identity and such that
\begin{equation} \label{BNFe}\tag{BNF} H \circ \Phi_m(\theta,I)=\omega\cdot I+H_m(I)+O\left(||I||^{m+1}\right) \end{equation}
where $H_m$ is a polynomial in $n$ variables of degree $m$ without constant or linear terms (see, for instance, \cite{Dou88}). One also have a formal symplectic transformation $\Phi_\infty$ and a formal series $H_\infty$ such that $H \circ \Phi_\infty(\theta,I)=H_\infty(I)$. The polynomials $H_m$ (respectively the formal series $H_\infty$) are uniquely defined and they are called the Birkhoff  polynomials (respectively Birkhoff formal series). 

Next we need to recall the following definition from~\cite{BFN15}, which was inspired by the work of Nekhoroshev (\cite{Nek73}). Let $P(n,m)$ be the space of polynomials with real coefficients of degree $m$ in $n$ variables, and $P_2(n,m)\subset P(n,m)$ the subspace of polynomials without constant or linear terms.

\begin{definition}[Stably steep polynomials] \label{stabsteep} 
Given positive constants $\rho$, $C$ and $\delta$, a polynomial $P_0 \in P_2(n,m)$ is called $(\rho,C,\delta)$-stably steep if for any integer $l\in [1,n-1]$, any $P \in P_2(n,m)$ such that $||P-P_0||<\rho$ and any vector subspace $\Lambda \subseteq \R^n$ of dimension $l$, letting $P_\Lambda$ be the restriction of $P$ to $\Lambda$, the inequality
\[ \max_{0 \leq \eta \leq \xi}\;\min_{||x||=\eta, \; x \in \Lambda}||\nabla P_\Lambda(x)||>C\xi^{m-1} \]
holds true for all $0 < \xi \leq\delta$. A polynomial $P_0 \in P_2(n,m)$ is called stably steep if there exist positive constants $\rho$, $C$ and $\delta$ such that $P_0$ is $(\rho,C,\delta)$-stably steep. 
\end{definition}

The set of stably steep polynomials in $P_2(n,m)$ will be denoted by $SS(n,m)$. 

For a fixed $\omega \in \R^n$, we let $\mathcal{H}^\alpha_L(\omega)$ be the space of Hamiltonians $H \in G^\alpha_L(\T^n \times B)$ as in~\eqref{H1}. We can now state our first main result. 

\begin{Main}\label{Th1}
Let $H \in \mathcal{H}^\alpha_L(\omega)$ with $\omega$ satisfying~\eqref{dio}. Assume that for $m_0:=[n^2/2+2]$, we have
\begin{equation}\label{condG}\tag{G}
H_{m_0} \in SS(n,m_0).
\end{equation}
Then ${\mathcal T}_{\o}$ is doubly exponentially stable with exponent $\frac{1}{\alpha(1+\tau)}$.
\end{Main}

The positive constants $r_*$ and $C$ of the double exponential stability depend only on $n$, $\gamma$ $\tau$, $\alpha$, $L$, $||H||_{\alpha,L,B}$ and on some constants characterizing the condition~\eqref{condG}.
 
It remains to explain in which sense~\eqref{condG} is generic. As we proved in~\cite{BFN15} (this will be recalled in \S\ref{s3}, Proposition~\ref{propNekho2}),  $SS(n,m_0)$ is an open set of full Lebesgue measure (and hence dense) in $P_2(n,m_0)$ (more precisely, its complement is a semi-algebraic subset of $P_2(n,m_0)$ of positive codimension). Using this, we proved in \cite{BFN15} that the Birkhoff normal form  of order $m_0$ of a Hamiltonian system at a non resonant elliptic fixed point is in general stably steep (\cite{BFN15}, Theorem C). In a similar fashion, we have here that {\it for any $H \in \mathcal{H}^\alpha_L(\omega)$ with $\omega$ satisfying~\eqref{dio}, for an open set of full Lebesgue measure of $Q \in P_2(n,m_0)$, the condition \eqref{condG} holds true for the modified Hamiltonian $H_Q \in \mathcal{H}^\alpha_L(\omega)$ defined by $H_Q(\theta,I):=H(\theta,I)+Q(I)$.} Indeed, the Diophantine condition on $\o$ is necessary just to be able to perform the Birkhoff normal form reduction up to order $m_0$ since we are dealing here with an invariant torus while in \cite{BFN15} we were interested with fixed points. Apart from this, the proof that the condition \eqref{condG} holds for $H_Q$  for an open set of full Lebesgue measure of $Q \in P_2(n,m_0)$ follows exactly  the same steps as in the fixed point case (see \cite{BFN15} Section \S 2.2). Hence we get the following corollary of Theorem \ref{Th1}.

\begin{Coro}\label{coro1} In the space of Hamiltonians $\mathcal{H}^\alpha_L(\omega)$ with $\omega$ satisfying~\eqref{dio}, the set of Hamiltonians for which ${\mathcal T}_{\o}$ is doubly exponentially stable with exponent $\frac{1}{\alpha(1+\tau)}$ contains an open dense and prevalent subset. 
\end{Coro}

\begin{proof} It follows directly from Theorem \ref{Th1} and the fact that the subspace of $\mathcal{H}^\alpha_L(\omega)$ for which~\eqref{condG} holds is open, dense and prevalent. Indeed, prevalence follows form the fact that~\eqref{condG} holds for $H_Q$ for a set of full Lebesgue measure of $Q \in P_2(n,m_0)$  (see \cite{HK10} for a survey on the notion of prevalence). Openness follows from the fact that~\eqref{condG} is an open condition, and density then follows from prevalence.   
\end{proof}

\bigskip

Now let us consider our second setting, that is we consider a Hamiltonian as in~\eqref{H2}. 

\begin{definition} We say that a completely integrable Hamiltonian $h\in G^\alpha_L(\bar{D})$ is {\it KAM doubly exponentially stable with exponent $u$} if for any $f \in G^\alpha_L(\T^n \times \bar{D})$, and for any $\gamma>0$, the following holds:  there exists $\varepsilon_*>0$ such that for any $\varepsilon\leq \varepsilon_*$, there   exists, for the Hamiltonian flow of $H_\varepsilon$ as in~\eqref{H2}, a set of invariant Lagrangian Diophantine tori $\mathcal{K}' \subseteq \T^n \times \bar{D}$ whose complement in $\T^n \times \bar{D}$ has a measure of order $\gamma$, and such that every torus in $\mathcal{K}' $ is doubly exponentially stable with exponent $u$.
\end{definition}

Given a smooth function $h$ defined on the closure $\bar{D}$ of some open bounded domain $D \subseteq \R^n$, for any $I\in \bar{D}$ and any integer $m\geq 2$, we define the Taylor polynomial $T_I^m h \in P_2(n,m)$ (starting at order 2) by
\[ T_I^m h(X)=\sum_{l=2}^{m}(l!)^{-1}\nabla^l h(I).(X)^l, \quad X \in \R^n\]
and recall the following definition.

\begin{definition}\label{defK}
Let us define by $N(n,2)$ the subset of $P_2(n,2)$ consisting of non-degenerate quadratic forms. A smooth function $h : \bar{D} \rightarrow \R$ is said to be Kolmogorov non-degenerate if for all $I \in \bar{D}$, $T_I^2h \in N(n,2)$. Upon restricting $\bar{D}$ if necessary, we may always assume that its gradient $\nabla h: \bar{D} \rightarrow \bar{\Omega}$ is a diffeomorphism onto its image. 
\end{definition}

Our second main result is then as follows.   

\begin{Main}\label{Th2} Let $h \in G^\alpha_L(\bar{D})$ such that for all $I \in \bar{D}$,
\begin{equation}\label{condK}\tag{K}
T_I^{2} h \in N(n,2)
\end{equation}
and
\begin{equation}\label{condS}\tag{S}
T_I^{m_0} h \in SS(n,m_0).
\end{equation}
Then $h$ is KAM doubly exponentially stable with exponent $u$, for any $u<\frac{1}{\a n}$.
\end{Main}

As mentioned in the introduction, under the Kolmogorov non-degeneracy condition on $h$, invariant tori with frequency in $\Omega_{\gamma,\tau}$ (see \eqref{def.Omega}) are preserved, being only slightly deformed, by an arbitrary $\varepsilon$-perturbation, provided $\varepsilon$ is sufficiently small. For Gevrey systems,  this was proved by Popov in \cite{Pop04}. It is only the second part of the statement (double exponential stability of the set of invariant tori) which is new, and the condition~\eqref{condS} is only required for this part of the statement. Under this condition, we will see in the proof of Theorem \ref{Th2} that for any fixed $\tau>n-1$, for any $\gamma>0$ and for $\varepsilon$ sufficiently small, the tori in the set of KAM invariant tori 
 \begin{equation} \label{def.tOmega} \mathcal K'=\bigcup_{\omega_{*} \in \Omega_{\gamma,\tau}'}\mathcal{T}_{\omega_{*}}^\varepsilon \end{equation}
are doubly exponentially stable with exponent $\frac{1}{\a(\tau+1)}$. 
Since $\tau$ can be chosen to be any number strictly smaller than $n-1$, we can reach any exponent $u$ with $u<\frac{1}{\a n}$. Here, we restricted, following \cite{Pop04}, the frequencies to the subset $\Omega_{\gamma,\tau}'$ of vectors in $\Omega_{\gamma,\tau}$ which have positive Lebesgue density in $\Omega_{\gamma,\tau}$; that is $\omega \in \Omega_{\gamma,\tau}'$ if for any neighborhood $O$ of $\omega$ in $\Omega$, the Lebesgue measure of $O \cap \Omega_{\gamma,\tau}$ is positive. No measure is lost due to this restriction since the sets $\Omega_{\gamma,\tau}$ and $\Omega_{\gamma,\tau}'$ have the same Lebesgue measure.

For every fixed couple $(\gamma,\tau)$, the positive constants $\varepsilon_*$, $r_*$ and $C$, characterizing the smallness of the perturbation as well as the constants that appear in the double exponential stability of the tori in $\mathcal K'$ of \eqref{def.tOmega},  depend only on $n$, $\gamma$, $\tau$, $\alpha$, $L$, $||h||_{\alpha,L,\bar{D}}$, $||f||_{\alpha,L,\bar{D}}$ and on some uniform constants characterizing the conditions~\eqref{condK} and~\eqref{condS}. The latter uniform constants can be obtained using  the compactness of $\bar{D}$ and  the fact that $N(n,2)$ and $SS(n,m_0)$ are open.


As before, it remains to explain in which sense the conditions~\eqref{condK} and~\eqref{condS} are generic. In fact we have the following lemma. 

\begin{lemma} \label{lemma.fub}  Fix $\gamma>0$. Given any $h \in G^\alpha_L(\bar{D})$, for an open set of full Lebesgue measure of $Q \in P_2(n,m_0)$, the modified integrable Hamiltonian $h_Q \in G^\alpha_L(\bar{D})$ defined by $h_Q(I):=h(I)+Q(I)$ satisfies~\eqref{condK} and~\eqref{condS} on a compact subset of $\bar{D}$ whose complement in $\bar{D}$ has a measure less than $\gamma$.
\end{lemma}

It immediately follows that the set of $Q  \in P_2(n,m_0)$ such that for any $\gamma>0$, $h_Q$  satisfies~\eqref{condK} and~\eqref{condS} (with constants that depend of course on $\gamma$) on a compact subset of $\bar{D}$ whose complement in $\bar{D}$ has  measure less than $\gamma$, contains a residual and prevalent set. Hence we get the following corollary of Theorem~\ref{Th2} and Lemma~\ref{lemma.fub}.

\begin{Coro}\label{coro2}  For a residual and prevalent set of  integrable Hamiltonians $h\in  G^\alpha_L(\bar{D})$, $h$ is  {\it KAM doubly exponentially stable with exponent $u$} for any $u<\frac{1}{\a n}$.
\end{Coro}  

Let us now give the proof of Lemma \ref{lemma.fub}. 

\begin{proof}[Proof of Lemma \ref{lemma.fub}]
Let us denote by $G_\gamma$ the set of $Q \in P_2(n,m_0)$ such that $h_Q$ satisfies both conditions~\eqref{condK} and~\eqref{condS} on a compact subset of $\bar{D}$ whose complement in $\bar{D}$ has a measure less than $\gamma$. What we need to prove is that $G_\gamma$ is both open and has full Lebesgue measure. 

Since the sets $N(n,2)$ and $SS(n,m_0)$ are open, by a compactness argument the set $G_\gamma$ is open, so it remains to prove that $G_\gamma$ has full Lebesgue measure.

Let us define the subset $KS \subset \bar{D} \times P_2(n,m_0)$ consisting of couples $(I,Q)$ for which the modified Hamiltonian $h_Q$ satisfies both conditions~\eqref{condK} and~\eqref{condS} at the point $I$. Let us also define the sections
\[ KS(I):=\{Q \in P_2(n,m_0) \; | \; (I,Q) \in KS \}, \quad KS(Q):=\{I \in \bar{D} \; | \; (I,Q) \in KS\} . \] 
For any $I \in \bar{D}$, we claim that the set $KS(I)$ has full Lebesgue measure. Indeed, it is obvious that the set of $Q  \in P_2(n,m_0)$ for which condition~\eqref{condK} is satisfied at $I$ has full Lebesgue measure while the fact that the set of $Q  \in P_2(n,m_0)$ for which condition~\eqref{condS} is satisfied at $I$ has full Lebesgue measure follows from~\cite{BFN15}. By Fubini's theorem, there exists a subset $G \subset P_2(n,m_0)$ of full Lebesgue measure such that for any $Q \in G$, the set $KS(Q)$ has full Lebesgue measure, that is $h_Q$ satisfies both conditions~\eqref{condK} and~\eqref{condS} at almost all points $I \in \bar{D}$. In particular, for any $Q \in G$ and for any $\gamma>0$, $h_Q$ satisfies both conditions~\eqref{condK} and~\eqref{condS} on a compact subset of $\bar{D}$ whose complement in $\bar{D}$ has a measure less than $\gamma$. This shows that $G$ is contained in $G_\gamma$ (in fact, $G$ is contained in the intersection of the $G_\gamma$ over $\gamma>0$), and therefore $G_\gamma$ has full Lebesgue measure.
\end{proof}

\subsection{Plan of the proofs}

The proofs of Theorem~\ref{Th1} and Theorem~~\ref{Th2}, although independent from each other, will follow the same path, which was the one also taken in~\cite{BFN15}. Let us first start with the case of a single Lagrangian invariant Diophantine torus.

The first step consists of constructing a Birkhoff normal form with an exponentially small remainder. In the analytic case, even though the Birkhoff formal series is expected to be divergent in general, its coefficients have a moderate growth of Gevrey type, so that on a ball of radius $r$, if one truncates the series at a large order $m \sim r^{-a}$, the non-integrable part in the normal form becomes exponentially small with respect to $m \sim r^{-a}$. In the case of an elliptic fixed point, these estimates have been proved but the calculations do not apply directly to the case of invariant tori. However a more general result (valid for Gevrey smooth systems) in the case of an invariant torus has been obtained by Mitev and Popov in~\cite{MP09}. Indeed, by Borel's theorem the formal Birkhoff series $h_{\infty}$ can be realized as the Taylor series at $I=0$ of some smooth function $H_*(I)$; it is proved in~\cite{MP09} that one can choose such a function $H_*$ in a Gevrey class, and this gives immediately the desired growth on the coefficients of $h_{\infty}$. More precisely, if $H$ is $\alpha$-Gevrey and $\omega$ satisfies~\eqref{dio}, then setting $\beta=\alpha(1+\tau)+1$, one can find a symplectic transformation $\Phi$ which is $(\alpha,\beta)$-Gevrey, such that $H \circ \Phi(\theta,I)=H_*(I)+R_*(\theta,I)$, where $H_*$ is $\beta$-Gevrey, $R_*$ is $(\alpha,\beta)$-Gevrey and flat at $I=0$. It follows that on a $r$-neighborhood of $I=0$, $R_*$ is exponentially small with respect to $r^{-(\beta-1)^{-1}}=r^{-\alpha^{-1}a}$. Now, if in addition, the Birkhoff polynomial $H_{m_0}$ is assumed to be stably steep, one can prove as in \cite{BFN15} that any sufficiently smooth function, whose Taylor expansion at some point coincides with $H_{m_0}$, is steep in a neighborhood of this point. By unicity of the Birkhoff polynomials, we then get that $H_*$ is steep in a neighborhood of the origin. So on a small $r$-neighborhood of $I=0$, $H \circ \Phi$ is an exponentially small (in $r$) perturbation of a steep integrable Hamiltonian: the double exponential stability follows by applying a version of Nekhoroshev's theorem in Gevrey classes.       

Now in the case of a family of KAM tori, under the Kolmogorov non-degeneracy condition~\eqref{condK}, Popov constructed in~\cite{Pop04} a simultaneous Birkhoff normal form for invariant tori with frequencies $\omega_* \in \Omega_{\gamma,\tau}'$: for $H_\varepsilon=h+\varepsilon f$, assuming that both $h$ and $f$ are $\alpha$-Gevrey, there exists a symplectic transformation $\Psi$ which is $(\alpha,\beta)$-Gevrey, such that $H_\varepsilon \circ \Psi(\theta,I)=h_*(I)+f_*(\theta,I)$, where $h_*$ is $\beta$-Gevrey, $f_*$ is $(\alpha,\beta)$-Gevrey, and $f_*$ is flat at every point $I \in (\nabla h_*)^{-1}(\Omega_{\gamma,\tau}')$. This immediately implies the persistence of invariant tori with frequencies in $\Omega_{\gamma,\tau}'$. Moreover, the derivatives of $h_*$ are close (with respect to $\varepsilon$) to the derivatives of $h$, hence the condition~\eqref{condS} also holds true with $h$ replaced by $h_*$, provided $\varepsilon$ is sufficiently small. This implies that $h_*$ is steep, while $f_*$ is exponentially small on a neighborhood of $\T^n \times (\nabla h_*)^{-1}(\Omega_{\gamma,\tau}')$: as before, the double exponential stability follows from Nekhoroshev theorem.     

The plan of the article is as follows. In \S\ref{s2}, we recall the results concerning the Birkhoff normal form for a Hamiltonian as in~\eqref{H1} and the simultaneous Birkhoff normal form for a Hamiltonian as in~\eqref{H2}, following~\cite{MP09} and~\cite{Pop04}. In \S\ref{s3}, we recall some results concerning stably steep polynomials, following  \cite{Nek73} and \cite{BFN15}. The proof of Theorem~\ref{Th1} and Theorem~\ref{Th2} will be given in \S\ref{s4}, using the results of \S\ref{s2} and \S\ref{s3}, and a version of Nekhoroshev estimates for perturbation of steep integrable Hamiltonians in Gevrey classes, that we state in Appendix~\ref{app}.  

\section{Birkhoff normal forms}\label{s2}

In the sequel, given some $\rho>0$ and some point $I_* \in \R^n$, we define $B_\rho(I_*)=V_\rho \{I_*\}$ the open ball of radius $\rho$ around $I_*$ in $\R^n$, and for $I_*=0\in \R^n$, we let $B_\rho=B_\rho(0)$.

\begin{theorem}\label{thmMP}
Let $H \in \mathcal{H}^\alpha_L(\omega)$ with $\omega$ satisfying~\eqref{dio}, and define $\beta:=\alpha(1+\tau)+1$. Then there exist positive constants $\bar{r}$, $L_1$, $L_2$ and $A$, which depend only on $n$, $\gamma$, $\tau$, $\alpha$, $L$ and $||H||_{\alpha,L,B}$ and a symplectic transformation
\[ \Phi : \T^n \times B_{\bar{r}} \rightarrow \T^n \times B \]
whose components belong to $G^{\alpha,\beta}_{L_1,L_2}(\T^n \times B_{\bar{r}})$, such that
\[ H \circ \Phi(\theta,I)=H_*(I)+R_*(\theta,I), \quad H_* \in G^{\beta}_{L_2}(B_{\bar{r}}), \quad R_* \in G^{\alpha,\beta}_{L_1,L_2}(\T^n \times B_{\bar{r}})  \]
and with the following properties:
\begin{itemize} 
\item[(1)] for any $(\theta,I) \in \T^n \times B_{\bar{r}}$ (resp. for any $(\theta',I') \in \Phi(\T^n \times B_{\bar{r}})$), we have $\Pi_I\Phi(\theta,I)=\mathrm{Id}+O\left(||I||^2\right)$ (resp. $\Pi_{I'}\Phi(\theta',I')=\mathrm{Id}+O\left(||I'||^2\right)$), where $\Pi_I$ (resp. $\Pi_{I'}$) denotes projection onto action space;
\item[(2)] the Taylor series of $H_*$ at $I=0$ is given by the Birkhoff formal series $H_{\infty}$;
\item[(3)] $\partial_I^l R_*(\theta,0)=0$ for all $\theta \in \T^n$ and all $l\in \N^n$: hence for any $r$ such that $0<2r < \bar{r}$,
\[ ||R_*||_{\alpha,\beta,L_1,L_2,B_{2r}} \leq A\exp\left(-(2L_2r)^{-\frac{1}{\alpha(1+\tau)}}\right). \]
\end{itemize}
\end{theorem}

This statement is a direct consequence of Theorem 2 in \cite{MP09}, to which we refer for a proof. The statement $(2)$ follows by construction of $H_*$, whereas $(3)$ follows from Stirling's formula. Indeed, the assumption that $R_*$ is $(\alpha,\beta)$-Gevrey and flat at $I=0$, implies, by Taylor's formula, that for any $m\in\N$:
\[ |\partial_\theta^k\partial_I^lR_*(\theta,I)|\leq ||R_*||_{\alpha,\beta,L_1,L_2,\bar{B}_{\bar{r}}}L_1^{|k|}L_2^{|l|}k!^\alpha l!^\beta m!^{\beta-1}(L_2||I||)^m. \]
Choosing
\[ m\sim (L_2||I||)^{-\frac{1}{\beta-1}}=(L_2||I||)^{-\frac{1}{\alpha(1+\tau)}}, \]
by Stirling's formula we eventually obtain
\[ |\partial_\theta^k\partial_I^lR_*(\theta,I)|\leq AL_1^{|k|}L_2^{|l|}k!^\alpha l!^\beta \exp\left(-(L_2\||I||)^{-\frac{1}{\alpha(1+\tau)}}\right)  \]
where $A$ depends only on $||R_*||_{\alpha,\beta,L_1,L_2,\bar{B}_{\bar{r}}}$ and $n$. For $||I||<2r$, this proves $(3)$.

\bigskip

Les us now give an analogous statement for a Hamiltonian as in~\eqref{H2}. 

\begin{theorem}\label{thmP}
Let $H_\varepsilon$ be as in~\eqref{H2}, with $h \in G^\alpha_L(\bar{D})$ and $f \in G^\alpha_L(\T^n \times \bar{D})$, and fix $\gamma>0$. Assume that~\eqref{condK} is satisfied. Then there exist positive constants $\bar{\varepsilon}$, $L_1'$, $L_2'$, $A'$ and $E$, which depend only on $n$, $\sigma$, $\gamma$, $\tau$, $\alpha$, $L$, $||h||_{\alpha,L,\bar{D}}$, $||f||_{\alpha,L,\bar{D}}$ and on the condition~\eqref{condK}, such that for $\varepsilon \leq \bar{\varepsilon}$, there exists a symplectic transformation
\[ \Psi : \T^n \times D \rightarrow \T^n \times D \]
whose components belong to $G^{\alpha,\beta}_{L_1',L_2'}(\T^n \times D)$, with $\beta=\alpha(1+\tau)+1$, such that
\[ H \circ \Psi(\theta,I)=h_*(I)+f_*(\theta,I), \quad h_* \in G^{\beta}_{L_2'}(D), \quad f_* \in G^{\alpha,\beta}_{L_1',L_2'}(\T^n \times D)  \]
and a diffeomorphism $\omega : D \rightarrow \Omega$ whose components belong to $G^{\beta}_{L_2'}(D)$, with the following properties: 
\begin{itemize}
\item[(1)] $||\Psi-\mathrm{Id}||_{C^1(\T^n \times D)}\leq E\sqrt{\varepsilon}$ and $||\Psi^{-1}-\mathrm{Id}||_{C^1(\T^n \times D)}\leq E\sqrt{\varepsilon}$, where $||\,\cdot\,||_{C^1(\T^n \times D)}$ denotes the $C^1$-norm on $\T^n \times D$;
\item[(2)] $||\partial_I^l\omega(I)-\partial_I^l\nabla h(I)||\leq EL_2'^{|l|}l!^\beta \sqrt{\varepsilon}$ for all $I \in D$ and $l \in \N^n$;
\item[(3)] setting $J_{\gamma,\tau}':=\omega^{-1}(\Omega_{\gamma,\tau}')$, we have $\partial_I^l\omega(I)=\partial_I^l\nabla h_*(I)$ for all $I \in J_{\gamma,\tau}'$ and $l \in \N^n$;   
\item[(4)] we have $\partial_I^l f_*(\theta,I)=0$ for all $(\theta,I)\in \T^n \times J_{\gamma,\tau}'$ and all $l\in \N^n$: hence for any $r>0$ sufficiently small so that $V_{2r}J_{\gamma,\tau}' \subseteq D$,
\[ ||f_*||_{\alpha,\beta,L_1',L_2',V_{2r}J_{\gamma,\tau}'} \leq A'\exp\left(-(2L_2'r)^{-\frac{1}{\alpha(1+\tau)}}\right). \]
\end{itemize}
\end{theorem}

This is exactly the content of Theorem 1.1 and Corollary 1.2 of~\cite{Pop04}, to which we refer for a proof.

\section{Stably steep polynomials}\label{s3} 

In this section, we recall some results concerning stably steep polynomials, which were defined in Definition~\ref{stabsteep}. The first one states that stably steep polynomials of sufficiently high degree are generic, and was used in the justification of Corollary~\ref{coro1} and Corollary~\ref{coro2}.  

\begin{proposition}\label{propNekho2}
The complement of $SS(n,m_0)$ in $P_2(n,m_0)$ is contained in a closed semi-algebraic subset of codimension at least one. In particular, $SS(n,m_0)$ is an open dense set of full Lebesgue measure.
\end{proposition} 

For a proof, we refer to Proposition A.2 and Proposition A.3 in~\cite{BFN15}. In order to state the next results, we first recall the definition of steep functions.

\begin{definition}\label{funcsteep}  
A differentiable function $h : G \rightarrow \R$ is steep on a domain $G' \subseteq G$ if there exist positive constants $C,\delta,p_l$, for any integer $l\in [1,n-1]$, and $\kappa$ such that for all $I \in G'$, we have $||\nabla h(I)|| \geq \kappa$ and, for all integer $l\in [1,n-1]$, for all vector space $\Lambda \in \R^n$ of dimension $l$, letting $\lambda=I+\Lambda$ the associated affine subspace passing through $I$ and $h_\lambda$ the restriction of $h$ to $\lambda$, the inequality
\[ \max_{0 \leq \eta \leq \xi}\;\min_{||I'-I||=\eta, \; I' \in \lambda}||\nabla h_\lambda(I')-\nabla h_\lambda(I)||>C\xi^{p_l} \]
holds true for all $0 < \xi \leq\delta$. We say that $h$ is $(\kappa,C,\delta,(p_l)_{l=1,\ldots,n-1})$-steep on $G'$ and, if all the $p_i=p$, we say that $h$ is $(\kappa,C,\delta,p)$-steep on $G'$.
\end{definition}

We have the following proposition, which states that if a smooth function has a stably steep Taylor polynomial at some point, then the function is steep on a neighborhood of this point.

\begin{proposition}\label{propsteep1}
Let $h : B_{\bar{r}} \rightarrow \R$ be a function of class $C^{m_0+1}$ such that $||\nabla h(0)||\geq\varpi>0$ and such that $T_0^{m_0}h$ is $(\rho',C',\delta')$-stably steep. Then for $r>0$ sufficiently small with respect to $\bar{r}, ||h||_{C^{m_0+1}(B_{\bar{r}})}, \rho', \varpi, m_0, C'$ and $\delta'$, the function $h$ is $(\kappa,C,\delta,m_0-1)$-steep on $B_{2r}$ with
\[ \kappa=\varpi/2, \quad C=C'/2, \quad \delta=r.  \] 
\end{proposition}

This proposition is analogous to Theorem 2.2 of~\cite{BFN15}; however, the setting being slightly different we give the details.

\begin{proof}[Proof of Proposition~\ref{propsteep1}]
We let $M=||h||_{C^{m_0+1}(B_{\bar{r}})}$, and we assume $3r<\bar{r}$. Then observe that for $r$ sufficiently small with respect to $M$ and $\varpi$, we have 
\[ ||\nabla h(I)||\geq \kappa=\varpi/2, \quad I \in B_{2r}.  \]
Since $P_0:=T_0^{m_0}h$ is $(\rho',C',\delta')$-stably steep, for $r$ sufficiently small with respect to $M$ and $\rho'$, $P_I:=T_I^{m_0}h$ is $(\rho'/2,C',\delta')$-stably steep for any $I \in B_{2r}$. By definition, for any vector subspace $\Lambda \subseteq \R^n$ of dimension $1 \leq l \leq n-1$, letting $P_{I,\Lambda}$ be the restriction of $P_I$ to $\Lambda$, the inequality
\begin{equation}\label{bout1}
\max_{0 \leq \eta \leq \xi}\;\min_{||x||=\eta, \; x \in \Lambda}||\nabla P_{I,\Lambda}(x)||>C'\xi^{m_0-1}
\end{equation}
holds true for all $0 < \xi \leq\delta'$. Assume $r\leq \delta'$, then for $x \in \Lambda$ such that $||x||\leq \xi \leq \delta=r$ and for any $I \in B_{2r}$, we have $I+x\in B_{3r} \subseteq B_{\bar{r}}$ hence by Taylor's formula (applied to $\nabla h$ at order $m_0$) we obtain
\[ ||\nabla h(I+x)-\nabla h(I)-\nabla P_I(x)||\leq Mm_0!||x||^{m_0}.  \]   
Therefore, if we assume $r$ sufficiently small with respect to $M$, $m_0$ and $C'$, we get
\[ ||\nabla h(I+x)-\nabla h(I)-\nabla P_I(x)||\leq (C'/2)||x||^{m_0-1}\leq (C'/2)\xi^{m_0-1}  \]
and hence, projecting the above vector onto $\Lambda$, we have
\begin{equation}\label{bout2}
||\nabla h_{\lambda}(I+x)-\nabla h_{\lambda}(I)-\nabla P_{I,\Lambda}(x)||\leq (C'/2)\xi^{m_0-1}
\end{equation}
where $\lambda=I+\Lambda$. Letting $I'=I+x$, the inequalities~\eqref{bout1} and~\eqref{bout2} yield, for any $I \in B_{2r}$, the inequality 
\begin{equation*}
\max_{0 \leq \eta \leq \xi}\;\min_{||I-I'||=\eta, \; I' \in \lambda}||\nabla h_{\lambda}(I')-\nabla h_{\lambda}(I)||> (C'/2)\xi^{m_0-1}=C\xi^{m_0-1}
\end{equation*}
for all $0 < \xi\leq \delta=r$. This concludes the proof. 
\end{proof}

We will also need a version of Proposition~\ref{propsteep1}, in which $I=0 \in \R^n$ is replaced by a compact set $K \subseteq \R^n$.

\begin{proposition}\label{propsteep2}
Let $K \subseteq \R^n$ be a compact set, $h : V_{\bar{r}}K \rightarrow \R$ be a function of class $C^{m_0+1}$ such that $||\nabla h(I)||\geq\varpi>0$ and $T_I^{m_0}h$ is $(\rho',C',\delta')$-stably steep for any $I \in K$. Then for $r>0$ sufficiently small with respect to $\bar{r}, ||h||_{C^{m_0+1}(V_{\bar{r}}K)}, \rho', \varpi, m_0, C'$ and $\delta'$, the function $h$ is $(\kappa,C,\delta,m_0-1)$-steep on $V_{2r}K$ with
\[ \kappa=\varpi/2, \quad C=C'/2, \quad \delta=r.  \] 
\end{proposition}

The proof is completely analogous to the proof of Proposition~\ref{propsteep1}, so we do not repeat the argument. 

\section{Proof of the main results}\label{s4}

In this section, we give the proofs of Theorem~\ref{Th1} and Theorem~\ref{Th2}. Theorem~\ref{Th1} follows from Theorem~\ref{thmMP}, Proposition~\ref{propsteep1} and a version of Nekhoroshev estimates for perturbations of steep integrable Hamiltonians in Gevrey classes, which is stated as Theorem~\ref{thmNek} in Appendix~\ref{app}. Theorem~\ref{Th2} follows from Theorem~\ref{thmP}, Proposition~\ref{propsteep2} and Theorem~\ref{thmNek}.

\begin{proof}[Proof of Theorem~\ref{Th1}]
Recall that we are considering $H \in \mathcal{H}^\alpha_L(\omega)$, with $\omega$ satisfying~\eqref{dio}, and that we are assuming that $H_{m_0} \in SS(n,m_0)$, for $m_0:=[n^2/2+2]$. We can apply Theorem~\ref{thmMP}: setting $\beta=\alpha(1+\tau)+1$, there exist positive constants $\bar{r},L_1,L_2,A$, which depend only on $n$, $\gamma$, $\tau$, $\alpha$, $L$ and $||H||_{\alpha,L,B}$ and a symplectic transformation
\[ \Phi : \T^n \times B_{\bar{r}} \rightarrow \T^n \times B \]
whose components belong to $G^{\alpha,\beta}_{L_1,L_2}(\T^n \times B_{\bar{r}})$, such that
\[ H \circ \Phi(\theta,I)=H_*(I)+R_*(\theta,I), \quad H_* \in G^{\beta}_{L_2}(B_{\bar{r}}), \quad R_* \in G^{\alpha,\beta}_{L_1,L_2}(\T^n \times B_{\bar{r}})  \]
and with the following properties:
\begin{itemize}
\item[(1)] for any $(\theta,I) \in \T^n \times B_{\bar{r}})$ (resp. for any $(\theta',I') \in \Phi(\T^n \times B_{\bar{r}})$), we have $\Pi_I\Phi(\theta,I)=\mathrm{Id}+O\left(||I||^2\right)$ (resp. $\Pi_{I'}\Phi(\theta',I')=\mathrm{Id}+O\left(||I'||^2\right)$), where $\Pi_I$ (resp. $\Pi_{I'}$) denotes projection onto action space;
\item[(2)] the Taylor series of $H_*$ at $I=0$ is given by the Birkhoff formal series $H_{\infty}$;
\item[(3)] $\partial_I^l R_*(\theta,0)=0$ for all $\theta \in \T^n$ and all $l\in \N^n$: hence for any $r$ such that $0<2r < \bar{r}$, 
\[ ||R_*||_{\alpha,\beta,L_1,L_2,\bar{B}_{2r}} \leq A\exp\left(-(2L_2r)^{-\frac{1}{\alpha(1+\tau)}}\right). \]
\end{itemize}
Let us choose $L_3 \geq \max\{L_1,L_2\}$. Then since $H_* \in G^{\beta}_{L_2}(B_{\bar{r}}) \subseteq G^{\beta}_{L_2}(B_{2r}) \subseteq G^{\beta}_{L_3}(B_{2r})$, we have
\[ ||H_*||_{\beta,L_3,B_{2r}}\leq F\]
for some positive constant $F$. Since $\beta>\alpha$, using $(3)$, we also have 
\[ ||R_*||_{\beta,L_3,\bar{B}_{2r}} \leq A\exp\left(-(2L_2r)^{-\frac{1}{\alpha(1+\tau)}}\right). \]
Then, using $(2)$, we have $T^{m_0}_0H_*=H_{m_0}$, hence   $T^{m_0}_0H_*\in SS(n,m_0)$ from our condition~\eqref{condG}. Together with the fact that $\nabla H_*(0)=\omega \neq 0$ this implies, by Proposition~\ref{propsteep1}, that for $r$ sufficiently small, the function $H_*$ is $(\kappa,C,r,m_0-1)$-steep on $B_{2r}$, for some positive constants $\kappa$ and $C$. We can therefore apply Theorem~\ref{thmNek} to the Hamiltonian $H \circ \Phi$, with
\[ \mu=A\exp\left(-(2L_2r)^{-\frac{1}{\alpha(1+\tau)}}\right). \]
Assuming $r>0$ sufficiently small, the threshold~\eqref{seuil} of Theorem~\ref{thmNek} is satisfied, and we find that for any solution $(\tilde{\theta}(t),\tilde{I}(t))$ of the Hamiltonian system associated to $H \circ \Phi$ with $\tilde{I}(0) \in B_{5r/4}$, then
\[ \tilde{I}(t)\in B_{3r/2}, \quad |t|\leq \exp\left(c''\mu^{-\frac{1}{2n\beta a}}\right). \]
Now observe that by choosing some positive constant $C$ sufficiently large, and requiring $r$ to be sufficiently small, we obtain
\[c''\mu^{-\frac{1}{2n\beta a}}=c''A^{-\frac{1}{2n\beta a}}\exp\left((2n\beta a)^{-1}(2L_2r)^{-\frac{1}{\alpha(1+\tau)}}\right) \geq \exp\left((Cr)^{-\frac{1}{\alpha(1+\tau)}}\right) \]
and thus we have in particular
\begin{equation}\label{estimf}
\tilde{I}(t)\in B_{3r/2}, \quad |t|\leq \exp\left(\exp\left((Cr)^{-\frac{1}{\alpha(1+\tau)}}\right)\right).
\end{equation}
To conclude the proof, consider an arbitrary solution $(\theta(t),I(t))$ of the Hamiltonian system associated to $H$, with $I(0) \in B_r$. Then $(\tilde{\theta}(t),\tilde{I}(t))=\Phi^{-1}(\theta(t),I(t))$ is a well-defined solution of $H \circ \Phi$, and using the property $(1)$ above, we can ensure that $\tilde{I}(0) \in B_{5/4r}$. Hence~\eqref{estimf} holds true, and using $(1)$ again, this implies in particular that
\begin{equation*}
I(t)\in B_{2r}, \quad |t|\leq \exp\left(\exp\left((Cr)^{-\frac{1}{\alpha(1+\tau)}}\right)\right)
\end{equation*}
which concludes the proof.
\end{proof} 

\begin{proof}[Proof of Theorem~\ref{Th2}]
Recall that we are considering $H_\varepsilon$ as in~\eqref{H2}, with $h \in G^\alpha_L(\bar{D})$ and $f \in G^\alpha_L(\T^n \times \bar{D})$, and that we are assuming that~\eqref{condK} and~\eqref{condS} are satisfied and that $\gamma>0$ is fixed. We can apply Theorem~\ref{thmP}: there exist positive constants $\bar{\varepsilon}$, $L_1'$, $L_2'$, $A'$, $E$, which depend only on $n$, $\sigma$, $\gamma$, $\tau$, $\alpha$, $L$, $||h||_{\alpha,L,\bar{D}}$, $||f||_{\alpha,L,\bar{D}}$ and on the condition~\eqref{condK}, such that for $\varepsilon \leq \bar{\varepsilon}$, there exists a symplectic transformation
\[ \Psi : \T^n \times D \rightarrow \T^n \times D \]
whose components belong to $G^{\alpha,\beta}_{L_1',L_2'}(\T^n \times D)$, with $\beta=\alpha(1+\tau)+1$, such that
\[ H \circ \Psi(\theta,I)=h_*(I)+f_*(\theta,I), \quad h_* \in G^{\beta}_{L_2'}(D), \quad f_* \in G^{\alpha,\beta}_{L_1',L_2'}(\T^n \times D)  \]
and a diffeomorphism $\omega : D \rightarrow \Omega$ whose components belong to $G^{\beta}_{L_2'}(D)$, with the following properties: 
\begin{itemize}
\item[(1)] $||\Psi-\mathrm{Id}||_{C^1(\T^n \times D)}\leq E\sqrt{\varepsilon}$ and $||\Psi^{-1}-\mathrm{Id}||_{C^1(\T^n \times D)}\leq E\sqrt{\varepsilon}$, where $||\,\cdot\,||_{C^1(\T^n \times D)}$ denotes the $C^1$-norm on $\T^n \times D$;
\item[(2)] $||\partial_I^l\omega(I)-\partial_I^l\nabla h(I)||\leq EL_2'^{|l|}l!^\beta \sqrt{\varepsilon}$ for all $I \in D$ and $l \in \N^n$;
\item[(3)] setting $J_{\gamma,\tau}'=\omega^{-1}(\Omega_{\gamma,\tau}')$, we have $\partial_I^l\omega(I)=\partial_I^l\nabla h_*(I)$ for all $I \in J_{\gamma,\tau}'$ and $l \in \N^n$;   
\item[(4)] we have $\partial_I^l f_*(\theta,I)=0$ for all $(\theta,I)\in \T^n \times J_{\gamma,\tau}'$ and all $l\in \N^n$: hence for any $r>0$ sufficiently small so that $V_{2r}J_{\gamma,\tau}' \subseteq D$,
\[ ||f_*||_{\alpha,\beta,L_1',L_2',V_{2r}J_{\gamma,\tau}'} \leq A'\exp\left(-(2L_2'r)^{-\frac{1}{\alpha(1+\tau)}}\right). \]
\end{itemize}
Then observe that since $SS(n,m_0)$ is open in $P_2(n,m_0)$, by compactness of $\bar{D}$ we can find positive constants $\rho'$, $C'$ and $\delta'$ so that for all $I \in \bar{D}$, $T_I^{m_0} h$ is $(\rho',C',\delta')$-stably steep. From $(2)$ and $(3)$ and for $\varepsilon$ small enough, we have that for $I \in J_{\gamma,\tau}'$, $T_I^{m_0} h_*$ is $(\rho'/2,C',\delta')$-stably steep. Using $(3)$ again, we have $\omega(I)=\nabla h_*(I)$ for all $I \in J_{\gamma,\tau}'$, therefore $\nabla h_*(J_{\gamma,\tau}')=\omega(J_{\gamma,\tau}')=\Omega_{\gamma,\tau}'$ is a compact subset of non-zero vectors and hence we can find $\varpi>0$ so that $||\nabla h_*(I)||\geq\varpi$ for any $I \in J_{\gamma,\tau}'$. We can apply Proposition~\ref{propsteep2} to $h_*$, with $K=J_{\gamma,\tau}'$, and for $r$ sufficiently small, the function $h_*$ is $(\kappa,C,r,m_0-1)$-steep on $V_{2r}J_{\gamma,\tau}'$, for some positive constants $\kappa$ and $C$. Then, since
\[ V_{2r}J_{\gamma,\tau}'=\bigcup_{I^* \in J_{\gamma,\tau}'}V_{2r}(\{I^*\})=\bigcup_{I^* \in J_{\gamma,\tau}'}B_{2r}(I^*), \]
by using $(4)$ we can apply Theorem~\ref{thmNek} exactly as in the proof of Theorem~\ref{Th1}, to arrive at the following statement: if $r>0$ is sufficiently small, for any $I^* \in J_{\gamma,\tau}'$ and any solution $(\tilde{\theta}(t),\tilde{I}(t))$ of the Hamiltonian system associated to $H \circ \Psi$, with $||\tilde{I}(0)-I^*|| \leq 5r/4$, we have
\begin{equation}\label{est}
||\tilde{I}(t)-I^*|| \leq 3r/2, \quad |t|\leq \exp\left(\exp\left((Cr)^{-\frac{1}{\alpha(1+\tau)}}\right)\right)
\end{equation}
for some positive constant $C$. Now, for $I_* \in J_{\gamma,\tau}'$, let $\omega_*=\omega(I_*)=\nabla h_*(I_*) \in \Omega_{\gamma,\tau}'$, and let us define
\[ \tilde{\mathcal{T}}_{\omega_*}^\varepsilon=\T^n \times {I^*}, \quad \tilde{\mathcal{K}}':=\T^n \times J_{\gamma,\tau}'=\bigcup_{\omega_* \in \Omega_{\gamma,\tau}'}\tilde{\mathcal{T}}_{\omega_*}^\varepsilon \subseteq \T^n \times D. \]
Then it follows easily from $(4)$ that $\tilde{\mathcal{K}}'$ is a set of Lagrangian Diophantine tori which is invariant by the Hamiltonian flow of $H \circ \Psi$, and whose complement in $\T^n \times D$ has a measure of order $\gamma$. Moreover, from~\eqref{est}, given any invariant torus $\tilde{\mathcal{T}}_{\omega_*}^\varepsilon$, any solution $(\tilde{\theta}(t),\tilde{I}(t))$ of the Hamiltonian system associated to $H \circ \Psi$, with $(\tilde{\theta}(0),\tilde{I}(0)) \in V_{5r/4}\tilde{\mathcal{T}}_{\omega_*}^\varepsilon$, satisfy
\begin{equation}\label{est2}
(\tilde{\theta}(t),\tilde{I}(t))\in V_{3r/2}\tilde{\mathcal{T}}_{\omega_*}^\varepsilon, \quad |t|\leq \exp\left(\exp\left((Cr)^{-\frac{1}{\alpha(1+\tau)}}\right)\right).
\end{equation}
Coming back to the original Hamiltonian $H$, we define
\[ \mathcal{T}_{\omega_*}^\varepsilon=\Psi(\tilde{\mathcal{T}}_{\omega_*}^\varepsilon), \quad \mathcal{K}':=\Psi(\tilde{\mathcal{K}}')=\bigcup_{\omega_* \in \Omega_{\gamma,\tau}'}\mathcal{T}_{\omega_*}^\varepsilon \subseteq \T^n \times D. \]
Since $\Psi$ is symplectic, $\mathcal{K}'$ is a set of Lagrangian Diophantine tori which is invariant by the Hamiltonian flow of $H$, and since $\Psi$ (as well as its inverse $\Psi^{-1}$) is close to the identity as expressed in $(1)$, the complement of $\mathcal{K}'$ in $\T^n \times D$ has also a measure of order $\gamma$. To conclude, using the estimate $(1)$ again and for $\varepsilon>0$ sufficiently small, $\Psi$ (as well as its inverse $\Psi^{-1}$) is Lipschitz with a Lipschitz constant sufficiently close to $1$, and so the inequality~\eqref{est2} implies the following statement: given any invariant torus $\mathcal{T}_{\omega_*}^\varepsilon$, any solution $(\theta(t),I(t))$ of the Hamiltonian system associated to $H \circ \Psi$, with $(\theta(0),I(0)) \in V_{r}\mathcal{T}_{\omega_*}^\varepsilon$, satisfy
\begin{equation*}
(\theta(t),I(t))\in V_{2r}\mathcal{\mathcal{T}_{\omega_*}^\varepsilon}, \quad |t|\leq \exp\left(\exp\left((Cr)^{-\frac{1}{\alpha(1+\tau)}}\right)\right).
\end{equation*}
Since we can choose any $\tau>n-1$, for any $u<\frac{1}{\a n}$ we have that each tori in the family $\mathcal{K}'$ is doubly-exponentially stable with exponent $u$, hence $h$ is KAM doubly exponentially stable with exponent $u$. This concludes the proof.
\end{proof} 

\appendix

\section{Nekhoroshev estimates for Gevrey steep Hamiltonian systems}\label{app}

The aim of this Appendix is to give a version of Nekhoroshev estimates for perturbations of steep integrable Hamiltonians in Gevrey classes. For real-analytic Hamiltonians, they were obtained by Nekhoroshev in his seminal works (\cite{Nek77}, \cite{Nek79}); here, for Gevrey smooth Hamiltonians we will follow the method of~\cite{Bou11cmp} and~\cite{BFN15}. 

Given $r>0$, let us denote by $B_{2r}(I_*)$ the ball of radius $2r$ around $I_*$ in $\R^n$ and consider a Hamiltonian as follows:
\begin{equation}\label{HamN}
\begin{cases}
H(\theta,I)=h(I)+f(\theta,I), \quad h : B_{2r}(I_*) \rightarrow \R, \quad f: \T^n \times B_{2r}(I_*) \rightarrow \R \\
||h||_{\beta,L,B_{2r}(I_*)} \leq F, \quad ||f||_{\beta,L,B_{2r}(I_*)} \leq \mu.   
\end{cases}
\end{equation}
Then we have the following result.

\begin{theorem}\label{thmNek}
Let $H$ be as in~\eqref{HamN}, and assume that $h$ is $(\kappa,C,r,p)$-steep on $B_{2r}(I_*)$. Then there exist positive constants $\mu_0$, $c$ $c'$ and $c''$, which depend only on $n$, $\beta$, $L$, $F$, $\kappa$, $C$ and $p$, such that if 
\begin{equation}\label{seuil}
\mu \leq \min\{\mu_0,cr^{2na}\},
\end{equation}
then for any solution $(\theta(t),I(t))$ of the Hamiltonian system associated to $H$ with $I(0) \in B_{5r/4}(I_*)$, we have
\[ ||I(t)-I(0)||\leq c'\mu^{\frac{1}{2na}}, \quad |t|\leq \exp\left(c''\mu^{-\frac{1}{2n \beta a}}\right) \]
where
\[a=1+p+p^2+\cdots+p^{n-1}.\]
In particular, we have
\[ I(t)\in B_{3r/2}(I_*), \quad |t|\leq \exp\left(c''\mu^{-\frac{1}{2n\beta a}}\right). \]
\end{theorem}

In the situations we are interested in, $r$ is a small parameter, and $\mu$ is exponentially small with respect to some power of $r^{-1}$ so that~\eqref{seuil} is automatically satisfied.  

The second part of the above statement is obviously a direct consequence of the first part. We will not prove the first part of Theorem~\ref{thmNek}, but we claim that it follows from the method of proof of~\cite{Bou11cmp}, Theorem 2.4 and~\cite{BFN15}, Theorem D. Indeed, the analytical part (the construction of resonant normal forms on suitable domains) can be taken from Theorem 2.4 of~\cite{Bou11cmp}, and combined with geometric part (the confinement argument) of Theorem D in~\cite{BFN15} to give a proof of Theorem~\ref{thmNek}. Let us recall that~\cite{BFN15} deals with elliptic fixed points in real-analytic systems; yet the geometric part is not sensible to the regularity of the system, and is the same in Cartesian or action-angle coordinates. Alternatively, one can only uses~\cite{Bou11cmp}, but at the expense of a worst exponent $a$ (which is, however, not important for our purpose here).

\addcontentsline{toc}{section}{References}
\bibliographystyle{amsalpha}
\bibliography{DoubleExpTore10}

\end{document}